\documentclass[11pt,reqno]{amsart}
\usepackage{amsmath,amsfonts, amssymb, amsthm}
  \usepackage{graphics} 
  \usepackage{epsfig} 
\usepackage{graphicx}  \usepackage{epstopdf}
 \usepackage[colorlinks=true]{hyperref}
\hypersetup{urlcolor=blue, citecolor=red}

\newtheorem{theorem}{Theorem}[section]
\newtheorem{corollary}{Corollary}[section]

\newtheorem{lemma}[theorem]{Lemma}

\theoremstyle{definition}

\newtheorem{remark}{Remark}[section]

\def\pmod #1{\ ({\rm{mod}}\ #1)}
\def\Z{\Bbb Z}
\def\N{\Bbb N}

\def\Q{\Bbb Q}

\def\R{\Bbb R}

\def\l{\left}
\def\r{\right}
\def\bg{\bigg}
\def\({\bg(}
\def\){\bg)}
\def\t{\text}
\def\f{\frac}

\def\mo{{\rm{mod}\ }}
\def\pmod#1{\ (\mo\ #1)}

\def\ls{\leqslant}
\def\gs{\geqslant}

\def\al{\alpha}

\def\ve{\varepsilon}

\def\eq{\equiv}

\def\da{\delta}

\def\Proof{\noindent{\it Proof}}

\begin{document}
\hbox{Preprint, {\tt arXiv:2605.05200}}
\medskip

\title[On a polynomial involving quadratic residues modulo primes]
      {On a polynomial involving quadratic residues modulo primes}
\author[Zhi-Wei Sun]{Zhi-Wei Sun}


\address{School of Mathematics, Nanjing
University, Nanjing 210093, People's Republic of China}
\email{{\tt zwsun@nju.edu.cn}
\newline\indent
{\it Homepage}: {\tt http://maths.nju.edu.cn/\lower0.5ex\hbox{\~{}}zwsun}}

\keywords{Polynomials, quadratic residues modulo primes, quadratic fields, roots of unity.
\newline \indent Supported by the Natural Science Foundation of China (grant no. 12371004).}

\begin{abstract} Let $p$ be an odd prime, and define
$$G_p(x)=\prod_{k=1}^{(p-1)/2}\left(x-e^{2\pi i k^2/p}\right).$$
In this paper we study values of $G_p(x)$ at roots of unity via Galois theory, and confirm some previous conjectures.
For example, for any primitive tenth root $\zeta$ of unity, we prove that
$$G_p(\zeta)=\begin{cases}(-1)^{|\{1\ls k\ls \frac {p+9}{10}:\ (\frac kp)=-1\}|}
&\text{if}\ p\equiv21\pmod{40},
\\(-1)^{|\{1\ls k\ls\frac {p+1}{10}:\ (\frac kp)=-1\}|}\zeta^{2}&\text{if}\ p\equiv 29\pmod{40},
\end{cases}$$
where $(\frac kp)$ denotes the Legendre symbol.
\end{abstract}
\maketitle

\section{Introduction}
\setcounter{equation}{0}
 \setcounter{conjecture}{0}
 \setcounter{theorem}{0}
 \setcounter{proposition}{0}

 Let $p$ be an odd prime, and let $(\f{\cdot}p)$ be the Legendre symbol.
 A classical result of Gauss (cf. \cite[pp.\,70-76]{IR}) states that
 $$\sum_{k=0}^{p-1}e^{2\pi i k^2/p}=\sqrt{(-1)^{(p-1)/2}p}.$$
 As in \cite{S23}, we define the polynomial
\begin{equation}\label{S-poly}G_p(x):=\prod_{k=1}^{(p-1)/2}(x-e^{2\pi ik^2/p})
=\prod_{r=1\atop (\f rp)=1}^{p-1}(x-e^{2\pi ir/p}).
\end{equation}
According to Dickson \cite[pp.\,370-371]{D99}, Dirichlet realized that
$(i-(\f 2p))G_p(i)\in\Z[\sqrt p]$ when $p\eq3\pmod4$.
The exact value of $G_p(\pm i)$ was determined by Williams \cite[Lemma 3]{W82}
in the case $p\eq3\pmod4$, and by the author \cite{S23} in the case $p\eq1\pmod4$.
The author \cite{S23} also determined $G_p(\omega)$ when $p\eq1\pmod4$, where
$$\omega=e^{2\pi i/3}=\f{-1+\sqrt3\,i}2.$$
In this paper, we continue to study values of the polynomial $G_p(x)$ at roots of unity.

Now we state our main results.

\begin{theorem} \label{Th1.1}
Let $p$ be an odd prime, and let $m>1$ be an integer with $p\nmid m$.
Let $a\in\Z$ with $0<a<m$.
If $p\eq1\pmod4$, then
\begin{equation}G_p(e^{2\pi ia/m})e^{-2\pi i\f am\cdot\f{p-1}4}(-1)^{|\{1\ls r<\f{ap}m:\ (\f rp)=1\}|}>0.
\end{equation}
When $p\eq3\pmod4$, we have
\begin{equation}G_p(e^{2\pi ia/m})ie^{-\pi i\f am\cdot\f{p-1}2}(-1)^{\f{h(-p)-1}2+|\{1\ls r<\f{ap}m:\ (\f rp)=1\}|}>0,
\end{equation}
where $h(-p)$ denotes the class number of the imaginary quadratic field $\Q(\sqrt{-p})$.
\end{theorem}
\begin{remark} This implies parts (iii) and (iv) of \cite[Conjecture 5.3]{S23} as well as
Conjectures 13.18, 13.19 and parts (ii)-(v) of Conjecture 13.20 in the book \cite{S-book}.
\end{remark}

\begin{theorem}\label{Th1.2} Let $p>5$ be a prime, and let $\zeta$ be any primitive tenth root of unity. Then
\begin{equation}G_p(\zeta)=\begin{cases}(-1)^{|\{1\ls k\ls \f {p+9}{10}:\ (\f kp)=-1\}|}
&\t{if}\ p\eq21\pmod{40},
\\(-1)^{|\{1\ls k\ls\f {p+1}{10}:\ (\f kp)=-1\}|}\zeta^{2}&\t{if}\ p\eq 29\pmod{40}.
\end{cases}\end{equation}
\end{theorem}
\begin{remark} This confirms \cite[Conjecture 5.2]{S23} (see also \cite[Conjecture 13.20(i)]{S-book}).
\end{remark}

\begin{theorem} \label{Th1.3} Let $p$ be any odd prime. Then there are $U_p(x),V_p(x)\in\Z[x]$ such that
\begin{equation}\label{GUV}G_p(x)=\f{V_p(x)+U_p(x)\sqrt{(-1)^{(p-1)/2}p}}2
\end{equation}
and
\begin{equation}\label{G*UV}G_p^*(x)=\f{V_p(x)-U_p(x)\sqrt{(-1)^{(p-1)/2}p}}2,
\end{equation}
where
\begin{equation}\label{G*}
G^*_p(x):=\prod_{r=1\atop (\f rp)=-1}^{p-1}\l(x-e^{2\pi ir/p}\r).
\end{equation}
Moreover, when $p>3$ we have
\begin{equation}\label{1/x}U_p(x^{-1})=x^{-(p-1)/2}U_p(x)
\ \ \t{and}\ \ V_p(x^{-1})=(-x)^{-(p-1)/2}V_p(x).
\end{equation}
\end{theorem}
\begin{remark} For any odd prime $p$, we obviously have
\begin{equation} \label{GG*} G_p(x)G^*_p(x)=\prod_{r=1}^{p-1}\l(x-e^{2\pi ir/p}\r)=\f{x^p-1}{x-1}.
 \end{equation}
\end{remark}

\begin{theorem} \label{Th1.4} Let $p>3$ be a prime with $p\equiv 3\pmod4$. Then
\begin{equation}\label{S(omega)}\begin{aligned}G_p(\omega)=&\ (-1)^{(h(-p)+1)/2}\left(\frac p3\right)\frac{x_p\sqrt3-y_p\sqrt{p}}2
\\&\ \times\begin{cases}i&\text{if}\ p\equiv7\pmod{12},
\\(-1)^{|\{1\le k<\frac p3:\ (\frac kp)=1\}|}i\omega&\text{if}\ p\equiv11\pmod{12},
\end{cases}\end{aligned}\end{equation}
and
\begin{equation}\label{S(omega-)}\begin{aligned}G_p(\bar\omega)=&\ (-1)^{(h(-p)-1)/2}\left(\frac p3\right)\frac{x_p\sqrt3+y_p\sqrt{p}}2
\\&\ \times\begin{cases}i&\text{if}\ p\equiv7\pmod{12},
\\(-1)^{|\{1\le k<\frac p3:\ (\frac kp)=1\}|}i\bar\omega&\text{if}\ p\equiv11\pmod{12},
\end{cases}\end{aligned}\end{equation}
where $(x_p,y_p)$ is the least positive integer solution to the diophantine equation
$$3x^2+4\left(\frac p3\right)=py^2.$$
\end{theorem}
\begin{remark} This confirms \cite[Conjecture 5.1]{S23}.
\end{remark}

As mentioned in \cite[Remark 5.2]{S23}, Theorem \ref{Th1.4}, together with \cite[Lemma 4.4]{S23},
yields the following result.

\begin{corollary} Let $p>3$ be a prime with $p\eq3\pmod4$. Then
\begin{align}G_p(-\omega)=\begin{cases}\omega&\t{if}\ p\eq11\pmod{24},
\\1&\t{if}\ p\eq 19\pmod{24},
\\(\f 3p)\omega^{(1+(\f 3p))/2}(x_p\sqrt3+y_p\sqrt p)^2/4&\t{if}\ p\eq 7\pmod{8},
\end{cases}\end{align}
where $x_p$ and $y_p$ are defined as in Theorem \ref{Th1.4}.
\end{corollary}

We will prove Theorems \ref{Th1.1} and \ref{Th1.2} in the next section.
Section 3 is devoted to our proofs of Theorems \ref{Th1.3} and \ref{Th1.4}.

\section{Proofs of Theorems \ref{Th1.1} and \ref{Th1.2}}
\setcounter{equation}{0}
 \setcounter{conjecture}{0}
 \setcounter{theorem}{0}
 \setcounter{proposition}{0}

\begin{lemma} \label{Lem2.1} Let $p>3$ be a prime, and let $R=\{1\ls r\ls p-1:\ (\f rp)=1\}|$. Then
\begin{equation}\label{r/p}\sum_{r\in R}\f rp =\begin{cases}(p-1)/4&\t{if}\ p\eq1\pmod4,
\\(p+1)/4-(h(-p)+1)/2&\t{if}\ p\eq3\pmod4.\end{cases}
\end{equation}
\end{lemma}
\Proof. Recall that $(\f{-1}p)=(-1)^{(p-1)/2}$ by the theory of quadratic residues modulo primes.

Suppose that $p\eq 1\pmod 4$. Then, for any $r\in\{1,\ldots,p-1\}$, we have $(\f {p-r}p)=(\f rp)$, and hence
$r\in R$ if and only if $p-r\in R$.
Thus
$$\sum_{r\in R}r=\sum_{r\in R}(p-r)=p|R|-\sum_{r\in R}r$$
and hence
$$\sum_{r\in R}r=\f p2|R|=\f{p(p-1)}4$$
as desired.

Now assume that $p\eq3\pmod4$. Then
$$2\sum_{r\in R}r=\sum_{r=1}^{p-1}r\l(1+\l(\f rp\r)\r)=\f{p(p-1)}2+\sum_{r=1}^{p-1}r\l(\f rp\r).$$
By Dirichlet's class number formula,
$$\sum_{r=1}^{p-1}r\l(\f rp\r)=-ph(-p).$$
Therefore
$$\f2p\sum_{r\in R}r=\f{p-1}2-h(-p)$$
and hence
$$\sum_{r\in R}\f rp=\f{p-1}4-\f{h(-p)}2=\f{p+1}4-\f{h(-p)+1}2$$
as desired.

In view of the above, we have completed the proof of Lemma \ref{Lem2.1}. \qed

\medskip
\noindent{\it Proof of Theorem \ref{Th1.1}}. Set $\zeta_n=e^{2\pi i/n}$ for $n\in\Z^+$.
For any $r\in\{1,\ldots,p-1\}$, we have
\begin{align*}\zeta_m^a-\zeta_p^r&=e^{i\pi(a/m+r/p)}\l(e^{i\pi (a/m-r/p)}-e^{i\pi(r/p-a/m)}\r)
\\&=e^{i\pi(a/m+r/p)}2i\sin \pi\l(\f am-\f rp\r).
\end{align*}
Thus
\begin{align*}G_p(\zeta_m^a)&=\prod_{r=1\atop(\f rp)=1}^{p-1}(\zeta_m^a-\zeta_p^r)
\\&=e^{i\pi(\f{p-1}2\cdot \f am+\sum_{r\in R}\f rp)}(2i)^{(p-1)/2}\prod_{r=1\atop (\f rp)=1}^{p-1}\sin\pi\l(\f am-\f rp\r),
\end{align*}
where $R=\{1\ls r\ls p-1:\ (\f rp)=1\}$. Note that
$$(-1)^{|\{\f{ap}m<r<p:\ (\f rp)=1\}|}\prod_{r=1\atop (\f rp)=1}^{p-1}\sin\pi\l(\f am-\f rp\r)
=\prod_{r=1\atop (\f rp)=1}^{p-1}\sin\pi\l|\f am-\f rp\r|>0.$$
Therefore
\begin{align*}&\ G_p(\zeta_m^a)i^{(p-1)/2}e^{-i\pi(\f{p-1}2\cdot\f am+\sum_{r\in R}\f rp)}(-1)^{|\{1\ls r<\f{ap}m:\ (\f rp)=1\}|}
\\=&\ 2^{(p-1)/2}\prod_{r=1\atop (\f rp)=1}^{p-1}\sin\pi\l|\f am-\f rp\r|
\\>&\ 0.
\end{align*}

If $p\eq1\pmod4$, then by Lemma \ref{Lem2.1} we have
$$i^{(p-1)/2}e^{-i\pi\sum_{r\in R}\f rp}=(-1)^{(p-1)/4}e^{-i\pi\f{p-1}4}=1.$$
If $p\eq3\pmod4$, then
by Lemma \ref{Lem2.1} we have
$$i^{(p-1)/2}e^{-i\pi\sum_{r\in R}\f rp}=\f{(i^2)^{(p+1)/4}}ie^{-i\pi(\f{p+1}4-\f{h(-p)+1}2)}=i(-1)^{(h(-p)-1)/2}.$$

Combining the above, we immediately obtained the desired result.
This completes the proof of Theorem \ref{Th1.1}. \qed

\begin{lemma} \label{Lem2.2} For any prime $p\eq5\pmod8$, we have
$G_p(x)\in\R[x]$ and
\begin{equation}\label{G(x^2)} G_p(x^2)=G_p^*(x)G_p^*(-x).
\end{equation}
\end{lemma}
\Proof. Let $\zeta_p=e^{2\pi i/p}$. As $(\f{-1}p)=1$, we have
$$G_p(x)=\prod_{r=1\atop(\f rp)=1}^{(p-1)/2}(x-\zeta_p^r)(x-\zeta_p^{p-r})
=\prod_{r=1\atop (\f rp)=1}^{(p-1)/2}\l(x^2-(\zeta^r+\zeta^{-r})x+1\r)\in\R[x].$$

Since $(\f 2p)=-1$ and
$|\{1\ls r\ls p-1:\ (\f rp)=-1\}|=\f{p-1}2\eq0\pmod2,$
 we have
$$G_p(x^2)=\prod_{r=1\atop (\f rp)=-1}^{p-1}(x^2-\zeta_p^{2r})
=\prod_{r=1\atop (\f rp)=-1}^{p-1}(x-\zeta_p^r)(x+\zeta_p^r)=G_p^*(x)G_p^*(-x).$$

In view of the above, we have proved the desired result. \qed

\medskip
\noindent{\it Proof of Theorem \ref{Th1.2}}. As $\zeta^5=-1$, by \eqref{G(x^2)} we have
\begin{equation}\label{216} G_p(\zeta^2)=G_p^*(\zeta)G_p^*(\zeta^6).
\end{equation}
Let $\zeta_{10}=e^{2\pi i/10}=e^{\pi i/5}$ and write $\zeta=\zeta_{10}^a$ with $a\in\{1,3,7,9\}$. By Theorem \ref{Th1.1},
\begin{equation}\label{G>0}
G_p(\zeta)\zeta^{-(p-1)/4}(-1)^{|\{1\ls r<\f{ap}{10}:\ (\f rp)=1\}|}>0,
\end{equation}
and also
\begin{equation}\label{R}G_p(\zeta^2)(\zeta^2)^{-(p-1)/4}\in\R
\ \ \t{and}\ \ \ G_p(\zeta^6)(\zeta^6)^{-(p-1)/4}\in\R.
\end{equation}
Obverse that
\begin{equation}\label{j14}\prod_{j=1}^4 G_p(\zeta^{2j})=\prod_{r=1\atop (\f rp)=1}^{p-1}
\prod_{j=1}^4(\zeta_p^r-e^{2\pi i j/5})
=\prod_{r=1\atop (\f rp)=1}^{p-1}\f{\zeta_p^{5r}-1}{\zeta_p^r-1}=1
\end{equation}
with the aid of the fact $(\f 5p)=(\f p5)=1$.

(i) We first handle the case $p\eq21\pmod{40}$.
As $p\eq1\pmod{10}$, we have $\zeta^p=\zeta$. In view of \eqref{GG*},
$$G_p(\zeta)G_p^*(\zeta)=\f{\zeta^p-1}{\zeta-1}=1
\ \t{and}\ G_p(\zeta^6)G_p^*(\zeta^6)=\f{\zeta^{6p}-1}{\zeta^6-1}=1.$$
Combining this with \eqref{216}, we obtain
\begin{equation}\label{G126}G_p(\zeta^2)G_p(\zeta)G_p(\zeta^6)=G_p^*(\zeta)G_p^*(\zeta^6)G_p(\zeta)G_p(\zeta^6)=1.
\end{equation}

 Note that $\zeta^{(p-1)/2}=1$ since $p\eq1\pmod{20}$.
So, by \eqref{R}, both $G_p(\zeta^2)$ and $G_p(\zeta^6)$ are real.
Recall that $G_p(x)\in \R[x]$ by Lemma \ref{Lem2.1}.
We now have
$$G_p(\zeta^2)=\overline{G_p(\zeta^2)}=G_p(\overline{\zeta^2})=G_p(\zeta^8)$$
and
$$G_p(\zeta^6)=\overline{G_p(\zeta^6)}=G_p(\overline{\zeta^6})=G_p(\zeta^4).$$
Therefore
\begin{align*}\l(G_p(\zeta^2)G_p(\zeta^6)\r)^2&=\prod_{j=1}^4 G_p(\zeta^{2j})=1
\end{align*}
in view of \eqref{j14}.
It follows that $G_p(\zeta^2)G_p(\zeta^6)\in\{\pm1\}$.
Combining this with \eqref{G126}, we get $G_p(\zeta)\in\{\pm1\}$.

As $(p-1)/4\eq 5\pmod{10}$, we have
 $\zeta^{-(p-1)/4}=\zeta^5=-1$ and hence
$$G_p(\zeta)(-1)^{|\{1\ls r<\f{ap}{10}:\ (\f rp)=1\}|}<0$$
by \eqref{G>0}. Since $G(\zeta)\in\{\pm1\}$, we must have
$$G_p(\zeta)=-(-1)^{|\{1\ls r<\f{ap}{10}:\ (\f rp)=1\}|}.$$
Note that
$$\l(\f{(p+9)/10}p\r)=\l(\f{10}p\r)=\l(\f2p\r)\l(\f 5p\r)=-1.$$
So we have
\begin{align*}G_p(\zeta_{10})&=-(-1)^{|\{1\ls r\ls \f{p+9}{10}:\ (\f rp)=1\}|}
=-(-1)^{\f{p+9}{10}-|\{1\ls r\ls \f{p+9}{10}:\ (\f rp)=-1\}|}
\\&=(-1)^{|\{1\ls r\ls \f{p+9}{10}:\ (\f rp)=-1\}|}.
\end{align*}

Let $\sigma_a$ be the element of the Galois group $\Q(\zeta_{10p})/\Q$
with
$$\sigma_a(e^{2\pi i(\f s{10}+\f rp)})=e^{2\pi i(\f{as}{10}+\f rp)}$$
for all $r,s\in\Z$. Applying the Galois automorphism $\sigma_{10}$ to the equality
$$G_p(\zeta_{10})=(-1)^{|\{1\ls r\ls \f{p+9}{10}:\ (\f rp)=-1\}|},$$
we obtain that
$$G_p(\zeta)=G_p(\zeta_{10}^a)=\sigma_a(G_p(\zeta_{10}))=(-1)^{|\{1\ls r\ls \f{p+9}{10}:\ (\f rp)=-1\}|}.$$

(ii) Now we handle the case $p\eq29\pmod{40}$.
As $p\eq9\pmod{10}$, we have $\zeta^p=\zeta^9=-\zeta^4$. In view of \eqref{GG*},
$$G_p(\zeta)G_p^*(\zeta)=\f{\zeta^p-1}{\zeta-1}=\f{\zeta^{-1}-1}{\zeta-1}=-\zeta^{-1}=\zeta^4$$
and
$$G_p(\zeta^6)G_p^*(\zeta^6)=\f{\zeta^{6p}-1}{\zeta^6-1}=-\zeta^{-6}=-\zeta^4.$$
Combining this with \eqref{216}, we obtain
\begin{equation}\label{G'126}G_p(\zeta^2)G_p(\zeta)G_p(\zeta^6)
=G_p^*(\zeta)G_p^*(\zeta^6)G_p(\zeta)G_p(\zeta^6)=-\zeta^8=\zeta^3.
\end{equation}

Note that $-(p-1)/4\eq3\pmod{10}$.
So, by \eqref{R}, both $G_p(\zeta^2)\zeta^6$ and $G_p(\zeta^6)\zeta^8$ are real.
Recall that $G_p(x)\in \R[x]$ by Lemma \ref{Lem2.1}.
We now have
$$G_p(\zeta^2)\zeta^6=\overline{G_p(\zeta^2)\zeta^6}=G_p(\zeta^8)\zeta^4$$
and
$$G_p(\zeta^6)\zeta^8=\overline{G_p(\zeta^6)\zeta^8}=G_p(\zeta^4)\zeta^2.$$
Hence
$$G_p(\zeta^2)=G_p(\zeta^8)\zeta^8\ \t{and}\ \ G_p(\zeta^6)=G_p(\zeta^4)\zeta^4.$$
Therefore
\begin{align*}\l(G_p(\zeta^2)G_p(\zeta^6)\r)^2\zeta^{-2}=\prod_{j=1}^4 G_p(\zeta^{2j})=1
\end{align*}
with the aid of \eqref{j14}.
It follows that $G_p(\zeta^2)G_p(\zeta^6)\zeta^{-1}\in\{\pm1\}$.
Combining this with \eqref{G'126}, we get $G_p(\zeta)/\zeta^2\in\{\pm1\}$.

As $-(p-1)/4\eq3\pmod{10}$, by \eqref{G>0} we have
$$G_p(\zeta)\zeta^3(-1)^{|\{1\ls r<\f{ap}{10}:\ (\f rp)=1\}|}$$
and hence
$$G_p(\zeta)\zeta^{-2}(-1)^{|\{1\ls r<\f{ap}{10}:\ (\f rp)=1\}|}<0.$$
Since $G_p(\zeta)\zeta^{-2}\in\{\pm1\}$, we must have
$$G_p(\zeta)\zeta^{-2}=-(-1)^{|\{1\ls r<\f{ap}{10}:\ (\f rp)=1\}|}.$$
In particular,
$$G_p(\zeta_{10})\zeta_{10}^{-2}=(-1)^{\f{p+1}{10}}(-1)^{|\{1\ls r\ls\f{p+1}{10}:\ (\f rp)=1\}|}.$$
(Note that $\f{p+1}{10}$ is a quadratic residue modulo $p$ since $(\f{10}p)=(\f 2p)(\f 5p)=-1$.)
Thus
$$G_p(\zeta_{10})\zeta_{10}^{-2}=(-1)^{|\{1\ls r\ls\f{p+1}{10}:\ (\f rp)=-1\}|}.$$
Applying the Galois automorphism $\sigma_a$ to this equality, we get
$$G_p(\zeta)\zeta^{-2}=(-1)^{|\{1\ls r\ls\f{p+1}{10}:\ (\f rp)=-1\}|}$$
as desired.

In view of the above, we have completed the proof of Theorem \ref{Th1.2}. \qed

\section{Proofs of Theorems \ref{Th1.3} and \ref{Th1.4}}
\setcounter{equation}{0}
 \setcounter{conjecture}{0}
 \setcounter{theorem}{0}
 \setcounter{proposition}{0}

\medskip
\noindent{\it Proof of Theorem \ref{Th1.3}}. Let $\zeta_p=e^{2\pi i/p}$.
The Galois group $G$ of the field extension $\Q(\zeta_p)/\Q$
consists of those $\sigma_a$ with $1\ls a\le p-1$ for which $\sigma_a(\zeta_p)=\zeta_p^a$.
Clearly,
 $G_p(x)=\sum_{j=0}^{(p-1)/2}c_jx^{(p-1)/2-j}$,
 where $c_0=1$ and
 $$c_j=(-1)^j\sum_{1\ls k_1<\ldots<k_j\ls(p-1)/2}\zeta_p^{k_1^2+\cdots+k_j^2}$$
 for all $j=1,\ldots,(p-1)/2$.
 For any $1\ls a\ls p-1$ with $(\f ap)=1$, clearly
 $$\prod_{k=1}^{(p-1)/2}(x-\zeta_p^{ak^2})=\prod_{r=1\atop (\f rp)=1}^{p-1}(x-\zeta_p^r)=G_p(x)$$
 and hence
 $$\sigma_a(c_j)=(-1)^j\sum_{1\ls k_1<\ldots<k_j\ls(p-1)/2}\zeta_p^{a(k_1^2+\cdots+k_j^2)}=c_j$$
 for all $0\ls j\ls (p-1)/2$. Thus
 $c_0,\ldots,c_{(p-1)/2}$ belong to the field
 $$M=\mathrm{Inv}(H)=\{\al\in \Q(\zeta_p):\ \sigma(\al)=\al\ \t{for all}\ \sigma\in H\},$$
 where $H=\{\sigma_a:\ 1\ls a\ls p-1\ \t{and}\ (\f ap)=1\}$
 is a subgroup of $G$ with $[G:H]=2$. By the Fundamental Theorem of Galois Theory, $\mathrm{Gal}(\Q(\zeta_p)/M)=H$ and
 $$[M:\Q]=|\mathrm{Gal}(M/\Q)|=|\mathrm{Gal}(\Q(\zeta_p)/\Q)/\mathrm{Gal}(\Q(\zeta_p)/M)|=|G/H|=2.$$

 By a well-known result on quadratic Gauss sums (cf. \cite[pp.\,70-76]{IR}, for each $a=1\,\ldots,p-1$
  we have
  $$\l(\f ap\r)\sum_{x=0}^{p-1}\zeta_p^{ax^2}=\sqrt{p'},$$
 where $p'=(-1)^{(p-1)/2}p$. Thus $\sqrt{p'}\in\Q(\zeta_p)$, and for any $1\ls a\ls p-1$ with $(\f ap)=1$ we have
 $$\sigma_a(\sqrt{p'})=\sum_{x=0}^{p-1}\zeta_p^{ax^2}=\sqrt{p'}.$$
 Therefore, $\sqrt{p'}\in\mathrm{Inv}(H)=M$. As $[M:\Q]=2$, we must have $M=\Q(\sqrt{p'})$.
 Note that $p'\eq1\pmod4$. It is well known that the ring $O_M$ of algebraic integers in $M$ consists of those numbers
 $(v+u\sqrt{p'})/2$ with $u,v\in\Z$ and $u\eq v\pmod2$.
 As $c_0,\ldots,c_{(p-1)/2}\in O_M$, we see that $G_p(x)=(V_p(x)+U_p(x)\sqrt{p'})/2$ for some
 $U_p(x),V_p(x)\in \Z[x]$.

 Now, let $a\in\{1,\ldots,p-1\}$ with $(\f ap)=-1$. Then
 $$\sum_{j=0}^{(p-1)/2}\sigma_a(c_j)x^{(p-1)/2-j}=\prod_{k=1}^{(p-1)/2}\l(x-\zeta_p^{ak^2}\r)
 =G_p^*(x).$$
 Write $c_j\in O_M$ as $(a_j+b_j\sqrt{p'})/2$ with $a_j,b_j\in\Z$ and $a_j\eq b_j\pmod2$.
 Since
 $$\sigma_a(\sqrt{p'})=\sum_{r=0}^{p-1}\zeta_p^{ar^2}=\l(\f ap\r)\sum_{r=0}^{p-1}\zeta_p^{r^2}=-\sqrt{-p},$$
 we have
 $\sigma(c_j)=(a_j-b_j\sqrt{p'})/2$. Thus
 $$G_p^*(x)=\sum_{j=0}^{(p-1)/2}\f{a_j-b_j\sqrt{p'}}2x^{(p-1)/2-j}=\f{V_p(x)-U_p(x)\sqrt{p'}}2.$$

 Now assume that $p>3$. Let
 $$R=\l\{1\ls r\ls p-1:\ \l(\f rp\r)=1\r\}\
 \t{and}\ N=\l\{1\ls r\ls p-1:\ \l(\f rp\r)=-1\r\}.$$
 Then
 $$\sum_{r\in R}r+\sum_{r\in N}r=\sum_{r=1}^{p-1}r=\f{p(p-1)}2\eq0\pmod p.$$
 Combining this with Lemma \ref{Lem2.1}, we get that
 $$\sum_{r\in R}r\eq0\eq\sum_{r\in N}r\pmod{p}.$$
 Hence
 \begin{equation}\label{Gx-1}G_p(x^{-1})=\prod_{r\in R}\l(-x^{-1}\zeta_p^r(x-\zeta_p^{-r})\r)=(-x)^{-(p-1)/2}\prod_{r\in R}(x-\zeta_p^{-r})
 \end{equation}
 and
 $$G_p^*(x^{-1})=\prod_{r\in N}\l(-x^{-1}\zeta_p^r(x-\zeta_p^{-r})\r)=(-x)^{-(p-1)/2}\prod_{r\in N}(x-\zeta_p^{-r}).$$
 Thus
 \begin{align*}V_p(x^{-1})&=G_p(x^{-1})+G_p^*(x^{-1})
 \\&=(-x)^{-(p-1)/2}\l(\prod_{r\in R}(x-\zeta_p^{-r})+\prod_{r\in N}(x-\zeta_p^{-r})\r)
 \\&=(-x)^{-(p-1)/2}(G_p(x)+G_p^*(x))
 \end{align*}
 and
 \begin{align*}U_p(x^{-1})&=\f{G_p(x^{-1})-G_p^*(x^{-1})}{\sqrt{p'}}
 \\&=\f{(-x)^{-(p-1)/2}}{\sqrt{p'}}\l(\prod_{r\in R}(x-\zeta_p^{-r})-\prod_{r\in N}(x-\zeta_p^{-r})\r)
 \\&=\f{(-x)^{-(p-1)/2}}{\sqrt{p'}}\l(\f {-1}p\r) (G_p(x)-G_p^*(x))
 =x^{-(p-1)/2}U_p(x).
 \end{align*}

 In view of the above, we have completed the proof of Theorem \ref{Th1.3}. \qed

 \begin{lemma} Let $p>3$ be a prime with $p\eq3\pmod4$. Then
 \begin{equation}\label{G-2}|G_p(\omega)|^{-2}=\ve_{3p}^{h(3p)},
 \end{equation}
 where $\ve_{3p}$ and $h(3p)$ are the fundamental unit and class number of the real
 quadratic field $\Q(\sqrt{3p})$, respectively.
 \end{lemma}
 \Proof. Note that $m=3p\eq1\pmod4$. By Dirichlet's class number formula (cf. \cite[p.\,344]{BS}),
 $$\prod_{r=1\atop (r,m)=1}^m(1-e^{2\pi i r/m})^{(\f rm)}=\ve_m^{-2h(m)}.$$
 Thus,
 \begin{align*}\ve_{3p}^{-2h(3p)}&=\prod_{r=1}^{p-1}\prod_{s=1}^2\l(1-e^{2\pi i(\f rp+\f s3)}\r)^{(\f{3r+ ps}{3p})}
\\&=\prod_{r=1}^{p-1}(1-\omega\zeta_p^r)^{(\f p3)(\f{3r}p)}(1-\omega^2\zeta_p^r)^{(\f{2p}3)(\f{3r}p)}
\\&=\prod_{r=1}^{p-1}(1-\omega\zeta_p^r)^{-(\f rp)}(1-\omega^2\zeta_p^r)^{(\f rp)}
\\&=\omega^{-\sum_{r=1}^{p-1}(\f rp)}(\omega^2-\zeta_p^r)^{-(\f rp)}
\times(\omega^2)^{\sum_{r=1}^{p-1}(\f rp)}\prod_{r=1}^{p-1}(\omega-\zeta_p^r)^{(\f rp)}
\\&=\prod_{r=1}^{p-1}(\omega-\zeta_p^r)^{(\f rp)}
\times\prod_{t=1}^{p-1}(\omega^2-\zeta_p^{p-t})^{-(\f{p-t}p)}
\\&=\prod_{r=1}^{p-1}(\omega-\zeta_p^r)^{(\f rp)}(\omega^{-1}-\zeta_p^{-r})^{(\f rp)}
=\prod_{r=1}^{p-1}|\omega-\zeta_p^r|^{2(\f rp)}
\end{align*}
and hence
\begin{align*}\ve_{3p}^{h(3p)}&=\prod_{r=1}^{p-1}|\omega-\zeta_p^r|^{-(\f rp)}
\\&=\f{\prod_{r\in R}|\omega-\zeta_p^r|\times\prod_{r\in N}|(\omega-\zeta_p^r)}
{\prod_{r\in R}|\omega-\zeta_p^r|^2}
\\&=\f{|G_p(\omega)G_p^*(\omega)|}{|G_p(\omega)|^2}.
\end{align*}
In light of \eqref{GG*},
 \begin{equation}\label{G*G}G_p(\omega)G_p^*(\omega)=\f{\omega^p-1}{\omega-1}=\f{\omega^{(\f p3)}-1}{\omega-1}=(i\omega)^{1-(\f p3)}.
 \end{equation}
 Therefore \eqref{G-2} holds. \qed
 
 \begin{lemma} Let $p>3$ be a prime with $p\eq3\pmod4$. Then
 \begin{equation}\label{even}\l|\l\{1\ls r\ls \l\lfloor\f {p-1}3\r\rfloor:\ \l(\f rp\r)=1\r\}\r|=\f{3+(\f p3)}4h(-p)+\f12\l\lfloor\f{p-1}3\r\rfloor.
 \end{equation}
\end{lemma}
\Proof. By Berndt \cite{Ber}, 
$$\sum_{r=1}^{\lfloor(p-1)/3\rfloor}\l(\f rp\r)=\l(3-\l(\f 3p\r)\r)\f{h(-p)}2=\l(3+\l(\f p3\r)\r)\f{h(-p)}2.$$
Thus
$$2\l|\l\{1\ls r\ls\l\lfloor\f{p-1}3\r\rfloor:\ \l(\f rp\r)=1\r\}\r|-\l\lfloor\f{p-1}3\r\rfloor
=\l(3+\l(\f p3\r)\r)\f{h(-p)}2$$
and hence \eqref{even} holds. \qed

 \medskip
 \noindent{\it Proof of Theorem \ref{Th1.4}}. By Theorem \ref{Th1.3}, there are $U_p(x),V_p(x)\in\mathbb Z[x]$ with $U_p(x^{-1})=x^{-(p-1)/2}U_p(x)$ and
 $V_p(x^{-1})=-x^{-(p-1)/2}V_p(x)$ such that
 $$G_p(x)=\f{V_p(x)+U_p(x)\sqrt{-p}}2
 \ \ \t{and}\ \ G_p^*(x)=\f{V_p(x)-U_p(x)\sqrt{-p}}2.$$
 Note that
 \begin{equation}\label{p/3}\f{V_p(\omega)^2+pU_p(\omega)^2}4=G_p(\omega)G_p^*(\omega)
 =(i\omega)^{1-(\f p3)}
 \end{equation}
 by \eqref{G*G}.

 As $p\eq(\f p3)\pmod 6$, we see that
 $$\omega^{-\f{p-1}2}=\omega^{-\f{p-(\f p3)}2+\f{1-(\f p3)}2}=\omega^{\f{1-(\f p3)}2}.$$
 Thus, by \eqref{1/x} we have
 $$\overline{U_p(\omega)}=U_p(\omega^{-1})=\omega^{(1-(\f p3))/2}U_p(\omega)$$
 and
 $$\overline{V_p(\omega)}=V_p(\omega^{-1})=-\omega^{(1-(\f p3))/2}V_p(\omega).$$
 So, if $p\eq1\pmod3$, then both $U_p(\omega)$ and $iV_p(\omega)$ are real numbers.
 Similarly, when $p\eq2\pmod3$, we have
 $$\overline{\omega^2U_p(\omega)}=\omega\overline{U_p(\omega)}=\omega^2U_p(\omega)
 \ \t{and}\
 \overline{i\omega^2V_p(\omega)}=-i\omega\overline{V_p(\omega)}=i\omega^2V_p(\omega),$$
 and hence both $\omega^{-1}U_p(\omega)$ and $i\omega^{-1}V_p(\omega)$ are real numbers.
Note that both $U_p(\omega)$ and $V_p(\omega)$ lie in the ring
 $$\Z[\omega]=\{a+b\omega:\ a,b\in\Z\}=\l\{a-\f b2+\f b2\sqrt3 i:\ a,b\in\Z\r\}.$$
 For $z=a+b\omega$ with $a,b\in\Z$, if $iz$ is real then $a=b/2$ and $z=a\sqrt3 i\in\sqrt{-3}\Z.$

 {\it Case} 1. $p\eq1\pmod3$.

 In this case, as $U_p(\omega)$ and $iV_p(\omega)$ are real numbers,
 we have $U_p(\omega)\in\Z$ and $V_p(\omega)\in\sqrt{-3}\Z$ by the above.
 In light of \eqref{p/3},
 $$3\l(\f{V_p(\omega)}{\sqrt{-3}}\r)^2+4=pU_p(\omega)^2.$$

 {\it Case} 2. $p\eq2\pmod3$.

 In this case, as both $\omega^{-1}U_p(\omega)$ and $i\omega^{-1}V_p(\omega)$ are real numbers, we have $U_p(\omega)/\omega\in\Z$ and $V_p(\omega)/\omega\in\sqrt{-3}\Z$.
 In light of \eqref{p/3},
 $$3\l(\f{V_p(\omega)}{\sqrt{-3}\omega}\r)^2-4=p\l(\f{U_p(\omega)}{\omega}\r)^2.$$

In either case,
$$x_0=\f{V_p(\omega)}{\sqrt{-3}}\omega^{1-(\f p3)}\in\Z
\ \t{and}\ y_0=-U_p(\omega)\omega^{1-(\f p3)}\in\Z,$$
and moreover $3x_0^2+4(\f p3)=py_0^2$. Clearly, both $x_0$ and $y_0$ are nonzero. Observe that
\begin{align*}G_p(\omega)&=\f{V_p(\omega)+U_p(\omega)\sqrt{-p}}2
=\omega^{(\f p3)-1}\f{x_0\sqrt{-3}-y_0\sqrt{-p}}2
\\&=i\omega^{(\f p3)-1}\f{x_0\sqrt3-y_0\sqrt p}2.
\end{align*}
As
$$(e^{-\pi i/3})^{(p-1)/2}=(-\omega)^{(p-1)/2}=-\omega^{((\f p3)-1)/2}=-\omega^{1-(\f p3)},$$
by Theorem \ref{Th1.1} we have
$$\f{x_0\sqrt3-y_0\sqrt p}2\da=G_p(\omega)(-i)\omega^{1-(\f p3)}\da>0$$
where
$$\da=(-1)^{\f{h(-p)-1}2+|\{1\ls r<\f p3:\ (\f rp)=1\}|}.$$
Let $x_1=-\da(\f p3) x_0$ and $y_1=-\da(\f p3)y_0$. Then
\begin{equation}
\label{x1y1}\l(\f p3\r)(y_1\sqrt p-x_1\sqrt3)>0\ \ \t{and}\ \ 3x_1^2+4\l(\f p3\r)=py_1^2.
\end{equation} Also,
\begin{equation}\label{G1}G_p(\omega)=-\da i\omega^{(\f p3)-1}\l(\f p3\r)\f{x_1\sqrt3-y_1\sqrt p}2
\end{equation}
and
$$|G_p(\omega)|^2=G_p(\omega)\overline{G_p(\omega)}=\l(\f{x_1\sqrt3-y_1\sqrt p}2\r)^2
=\l(\f{x_1\sqrt3+y_1\sqrt p}2\r)^{-2}.$$

In view of \eqref{G-2},
$$\l(\f{x_1\sqrt3+y_1\sqrt p}2\r)^2=|G_p(\omega)|^{-2}=\ve_{3p}^{h(3p)}>1$$
 and hence
$$\l|\f2{x_1\sqrt3-y_1\sqrt p}\r|=\l|\f{x_1\sqrt3+y_1\sqrt p}2\r|>1.$$
Thus we must have $x_1y_1>0$, and hence $x_1,y_1\in\Z^+$ in light of \eqref{x1y1}.
Note that $\beta_1=(x_1\sqrt3+y_1\sqrt p)/2>1$ and $\beta_1^2=\ve_{3p}^{h(3p)}$.

As $3$ divides the discriminant $3p$ of the field $K=\Q(3p)$, the prime $3$ ramifies and hence
$(3)=P^2$ for some prime ideal $P$ of $O_K$ (which is the ring of algebraic integers in $K$).
Note that $N_{K/\Q}(P)^2=3^2$ and hence $N_{K/\Q}(P)=3$.
For integers $x,y\in\Z$ with $x\eq y\pmod2$, clearly $3x^2+4(\f p3)=py^2$ if and only if
$N_{K/\Q}(\al)=-3(\f p3)$, where $\al=(3x+y\sqrt{3p})/2\in O_K$.
If $\al\in O_K$ with $N_{K/\Q}(\al)|=3=N(P)$, then we must have $P=(\al)=\al O_K$.
As $\al_1=(3x_1+y_1\sqrt{3p})/2\in O_K$ and $|N_{K/\Q}(\al_1)|=3$, we have $P=(\al_1)$.

Let $(x_p,y_p)$ be the smallest positive integer solution to the diophantine equation
$3x^2+4(\f p3)=py^2$. Then $x_p\ls x_1$ and $y_p\ls y_1$.
Also, $\al_p=(3x_p+y_p\sqrt{3p})/2>1$ is a generator of the prime ideal $P$.
For
$$\beta_p=\f{\al_p}{\sqrt3}=\f{x_p\sqrt3+y_p\sqrt p}2>1,$$
we have
$$\beta_p^2=\f{3x_p^2+py_p^2+2x_py_p\sqrt{3p}}4=\f{py_p^2-2(\f p3)+x_py_p\sqrt{3p}}2\in O_K$$
and
$$N_{K/\Q}(\beta_p^2)=\f{py_p^2(py_p^2-4(\f p3))+4-3px_p^2y_p^2}{4}=1.$$
So, $\beta_p^2=\ve_{3p}^m$ for some positive integer $m$. As $\beta_p^2\ls \beta_1^2=\ve_{3p}^{h(3p)}$,
we have $m\ls h(3p)$. As $\al_1\gs \al_p$ and $(\al_1)=P=(\al_p)$, $\al_1/\al_p$
is a unit in $O_K$ not smaller than one. Thus, $\al_1/\al_p=\ve_{3p}^k$ for some $k\in\N$.
Hence
$$\f{\beta_1^2}{\beta_p^2}=\l(\f{\al_1/\sqrt3}{\al_p/\sqrt3}\r)^2=\ve_{3p}^{2k}$$
and
$$\ve_{3p}^{h(3p)}=\beta_1^2=\beta_p^2\ve_{3p}^{-2k}=\ve_{3p}^{m+2k}.$$
Thus $m$ is odd since $h(3p)$ is known to be odd (cf. \cite{RR}).
Since $\ve_{3p}^m=\beta_p^2=\al_p^2/3$ generates a subgroup of the unit group containing the cyclotomic unit $\ve_{3p}^{h(p)}=\beta_1^2=\al_1^2/3$, and $(\al_p)=(\al_1)=P$,
by the Sinnott-Hasse Index Theorem  for cyclotomic units in real quadratic fields, $m$ cannot be smaller than $h(3p)$. Therefore, $m=h(3p)$ and $\al_1=\al_p$. Thus $x_1=x_p$ and $y_1=y_p$.
Combining this with \eqref{G1}, we obtain
$$G_p(\omega)=(-1)^{\f{h(-p)+1}2+|\{1\ls r<\f p3:\ (\f rp)=1\}|}i\omega^{(\f p3)-1}\l(\f p3\r)\f{x_p\sqrt3-y_p\sqrt p}2.$$
When $p\eq7\pmod{12}$, by \eqref{even} we have
$$\l|\l\{1\ls r\ls\f{p-1}3:\ \l(\f rp\r)=1\r\}\r|=h(-p)+\f{p-1}6\eq0\pmod2$$
since $h(-p)$ is odd. So 
\eqref{S(omega)} holds no matter $(\f p3)$ is $1$ or $-1$.

By \eqref{Gx-1} and \eqref{p/3},
$$G(\omega^2)=(-\omega)^{(p-1)/2}G_p^*(\omega)=-\omega^{(\f p3)-1}G_p^*(\omega)$$
and $G_p(\omega)G_p^*(\omega)=(i\omega)^{1-(\f p3)}.$
Thus
$$G_p(\omega^2)=-\omega^{(\f p3)-1}\f{(i\omega)^{1-(\f p3)}}{G_p(\omega)}=-\l(\f p3\r)G_p(\omega)^{-1}.$$
Note also that
$$\l(-\l(\f p3\r)\f{x_p\sqrt3-y_p\sqrt p}2\r)^{-1}=\f{x_p\sqrt3+y_p\sqrt p}2.$$
So we can easily deduce \eqref{S(omega-)} from \eqref{S(omega)}.
This ends our proof of Theorem \ref{Th1.4}. \qed

\end{document}